\documentclass[11pt]{amsart}
\usepackage{a4wide}
\usepackage{amsmath}
\usepackage{amssymb,amsthm,amsfonts}
\usepackage{amsrefs}
\usepackage[utf8]{inputenc}
\usepackage[T1]{fontenc}
\usepackage{xcolor}
\usepackage[colorlinks]{hyperref}

\newcommand{\Z}{{\mathbb Z}}
\newcommand{\N}{{\mathbb N}}
 
\newcommand{\R}{{\mathbb R}}

\newcommand{\Per}{\mathrm{Per}}

\newtheorem{theorem}{Theorem}[section]

\newtheorem{lemma}[theorem]{Lemma}
\newtheorem{proposition}[theorem]{Proposition}

\theoremstyle{definition}
\newtheorem{definition}[theorem]{Definition}
\newtheorem{remark}[theorem]{Remark}

\begin{document}
 
\title{Minimal periodic foams with fixed inradius}

\author{Annalisa Cesaroni}
\address{Dipartimento di Matematica Tullio Levi-Civita, Universit\`{a} di Padova, Via Trieste 63, 35131 Padova, Italy}
\email{annalisa.cesaroni@unipd.it}
\author{Matteo Novaga}
\address{Dipartimento di Matematica, Universit\`{a} di Pisa, Largo Bruno Pontecorvo 5, 56127 Pisa, Italy}
\email{matteo.novaga@unipi.it}

 \begin{abstract} 
 In this note we show existence and regularity of 
periodic tilings of the Euclidean space into equal cells containing a ball of fixed radius, which minimize either
 the classical or the fractional perimeter. We also  discuss some qualitative properties of minimizers in dimensions $3$ and $4$.  
\end{abstract} 
  
 
\maketitle 
 
\tableofcontents
 
\section{Introduction}
An important question, which has been considered since ancient times, is finding a partition of $\R^2$ into cells of equal area and with minimal perimeter, see \cite{morgansolo}.
This problem
was eventually solved by Hales in \cite{hales}, who proved that the hexagonal honeycomb is the unique minimizer.
The analogous problem 
in dimension $N\ge 3$ is commonly known as Kelvin problem; the name comes from  
a famous conjecture due to Lord Kelvin \cite{thompson}, who claimed that in a tiling of $\R^3$
with cells of unit volume, the region with minimal surface area is given by a
slightly modified truncated octahedron, also called Kelvin cell or tetrakaidecahedron.
This conjecture was recently disproved by Weaire and Phelan, see   in \cite{w} and the discussion in Section \ref{kfoam}.  
The Kelvin problem remains open for every $N\ge 3$, even if one restricts to lattice tilings, that is,
to periodic partitions where all the cells are  equal and given by a fundamental domain of a lattice.
The counterexample by Weaire and Phelan does not apply to this case, so that the Kelvin foam could be a minimizer in $\R^3$.

We notice that the lattice periodic Kelvin problem and, more generally, the 
  isoperimetric problem for fundamental domains of a closed manifold, has been considered in the literature, starting from the paper \cite{choe}, see also \cite{mnpr}.
   The concentration compactness argument for infinite partitions has been exploited in \cite{npst, npst2, cn}, whereas general perimeter
    functionals of local and nonlocal type have been considered in \cite{cfn, cn, cnproc}, see also \cite{nn} for a variant without periodicity.
  
In this note we  consider an isoperimetric problem which is 
related to the lattice periodic Kelvin problem.  In particular we fix a perimeter functional, which is the classical perimeter or the fractional perimeter, and we show existence of periodic  partitions of $\R^N$ in  equals cells with minimal perimeter,
under the constraint that each cell contains a ball of fixed radius. 
We prove that this problem admits always a solution, by combining a concentration compactness argument (see \cite{alm, maggi}),  and a compactness theorem for lattices due to Mahler, see \cite{ma}. We provide also some basic regularity property for the minimizers, based on the observation that the partition of the space generated by a minimal cell is locally (almost) minimizing the perimeter, up to compact sufficiently small  perturbations, see Section \ref{secreg}.



The problem  we study in this note   has been considered in dimension $N=3$, restricting the class of admissible cells to parallelohedra, that is, adding a convexity constraint, see \cite{bez, bez2, langi}. The identification of the shape of the solution is still open, but in this direction Bezdek, see \cite{bez3,bez, bez2}  stated the Rhombic  Dodecahedral  Conjecture: 
{\it the perimeter of  Voronoi cells  associated to a lattice packing of spheres of radius $1$ in $\R^3$   is at least that of a  rhombic dodecahedron.} Recall that the rhombic dodecahedron is    the Voronoi cell of the face-centered cubic lattice. 
\\  
A proof of this conjecture is still missing, though in \cite{bez3} it is shown that the rhombic dodecahedron is a local minimizer for this problem, up to local perturbations of   the face-centered cubic lattice.  Moreover   Hales showed  in \cite{hales} that  the perimeter of Voronoi cell associated to a packing of spheres  of radius $1$ in $\R^3$ is always bigger than the perimeter   of the regular dodecahedron, which however is not a parallelohedron (so it is not a Voronoi cell).

 Nonetheless the  rhombic dodecahedron  cannot be a solution of our isoperimetric problem, since it does not satisfy the necessary Plateau's conditions for  local minimality, see Proposition \ref{nomin}. So, apart from the dimension $N=2$, the problem of identifying the shape of the minimizer remains open, even if one may formulate some conjectures, at least in dimension $N=3, 4$, see Sections \ref{kfoam}, \ref{24cell}. 

Eventually we mention a natural extension of our problem, dropping the periodicity condition, which is briefly discussed in \cite[p. 25]{bez 2}, where the author poses the following question: {\it if the Euclidean 3-space is partitioned into cells each containing a unit ball, how
should the shapes of the cells be designed to minimize the average surface area of the cells?}\\
He also observes that the regular rhombic dodecahedron is not a solution of this general problem (see the discussion in Section \ref{kfoam}).

 \smallskip
 The paper is organized as follows. In Section \ref{seclattice} we recall the basic definitions and properties of lattices in $\R^N$ and we provide some important examples. 
 Section \ref{seciso} contains the main result of the paper, that is the existence of the solution of the isoperimetric problem, whereas in Section \ref{secreg} we provide some regularity properties  of the solution. Finally Sections \ref{kfoam} and \ref{24cell} are  devoted to a discussion  of the case of dimensions $N=3$ and $N=4$ respectively. 
 
 \smallskip

{\it Acknowledgements.} The first author was supported by the PRIN Project 2022W58BJ5; 
the second author was supported by the PRIN 2022 Project 2022E9CF89
and by the MUR Excellence Department Project awarded to the Department of Mathematics of the University of Pisa.
The authors are members of INDAM-GNAMPA.

\smallskip
\section{Lattices} \label{seclattice}

A lattice is a discrete subgroup $G$ of $(\R^N, +)$ of rank $N$. 
The elements of $G$ can be expressed as $\sum k_iv_i$, for a given basis  $(v_1, \dots, v_N)$ of $\R^N$, with    coefficients $k_i\in \Z$.   Any two bases for a lattice $G$ are related  by a matrix with integer coefficients and determinant equal to $\pm1$.  Equivalently, every lattice can be viewed as a  discrete group of  isometries of $\R^N$.

The absolute value of the determinant of the matrix of any set of generators $v_i$ of a lattice is uniquely determined and it is equal to $d(G)\in (0,+\infty)$, which we call {\it volume of the lattice }$G$. Equivalently, if we interpret $G$ as a discrete group of isometries,   $d(G)$ coincides with the volume of the quotient torus $\R^N/G$. 
 
We recall this characterization of lattices, see \cite{cn, cfn}. 
\begin{lemma} \label{lemmaretta} 
A closed subgroup $G$ of $(\R^N, +)$ is discrete if and only if it does not contain a line.
\end{lemma}

%
We let $\lambda(G)>0$  be the minimal norm of the nonzero elements of $G$. 
In particular, for every $p,q\in G$, there holds that $|p-q|\geq {\lambda (G)}$. 
Other important values associated to a lattice $G$  are its {\it inradius} $\rho_G$  and its {\it covering radius} $r_G$ defined  as 
\begin{eqnarray*} 
\rho_G&:=&\sup\{ r\ :\  \forall x\neq y\in G, B_r(x)\cap B_r(y)= \emptyset\}= \frac{{\lambda(G)}}{2}\\\nonumber
 r_G&:=&\inf\{ r\ :\  G+B_r=\R^N\}.\end{eqnarray*}
 
 The inradius is related to the problem of optimal packing of spheres, 
which can be  stated as follows: find  the  arrangement of non-overlapping identical spheres (of radius $r$) in $\R^N$  which maximizes the  packing density,  that is  the  proportion of space filled by the spheres.  A lattice arrangement is an arrangement of spheres centered at the points of the lattice and with radius given by the inradius.  
The covering radius is related to the dual problem of the optimal packing, that is finding  the most economical way to cover the space $\R^N$ with  overlapping spheres (of radius $r$). For more details on such problems we refer to the monograph \cite[Chapters 1,2]{CS}.  We recall that in dimension $2$   the arrangement of spheres which solves all these problems is  the regular  hexagonal lattice arrangement, see Remark \ref{hex}.

 A lattice $G$ is symmetric if it has a basis $(v_1, \dots, v_N)$ of $\R^N$ of generators such that $|v_i|= {\lambda(G)}$. 
 It is clear that there exist non symmetric lattices, nonetheless it is always possible to choose   a reduced set of generators for lattices (see \cite[Theorem 1]{ma}), as follows:
 
 \begin{lemma}\label{limi} There exists a dimensional constant $C_N$ such that every lattice $G$ admits a set of generators $v_1, \dots, v_n$ 
 with $\Pi_{i=1}^N |v_i|\leq C_N d(G)$. 
 \end{lemma}  
In the following we present some examples of well known lattices (which are all symmetric).

\begin{remark}[Root lattice and dual] 
The root lattice $A_N$ is a $N$ dimensional lattice which lies in the hypersurface $\sum_{i=1}^{N+1}x_1=0$ of $\R^{N+1}$ It is defined as 
\[A_N=\{z\in \Z^{N+1}, \sum_i z_i=0\}.\] The minimal vectors are permutations of $(1, -1, 0\dots, 0)$ and the norm is $\sqrt{2}$, so the inradius $\rho_{A_N}=\frac{\sqrt{2}}{2}$.

The dual of the root lattice is $A_N^*$ defined as $A_N^*=\cup_{i=0}^N A_N+[i]$, where $[i]\in \R^{n+1}$ has the first  $N+1-i$ components equal to $\frac{i}{N+1}$ and the last $i$ components equal to $\frac{i}{N+1}-1$.  The inradius of this lattice is $\rho_{A_N^*} = \frac{1}{2}\sqrt{\frac{N}{N+1}}$. 
\end{remark}

\begin{remark}[Checkboard lattice]\label{check} The checkboard lattice $D_N$ is given by $\{z\in \Z^N, \sum_i z_i \text{ is even}\}$. It's inradius is $\rho_{D_N}=\frac{\sqrt{2}}{2}$.  The covering radius is $r_{D_N}=\frac{\sqrt{N}}{2}$ for $N\geq 4$.  

When $N= 8$ the covering radius equal the minimum distance between two points of the lattice, and it is possible to add another copy of the lattice, so for $N\geq 8$ we define
\[D_N^+=D_N\cup \left(D_N+\left(\frac12, \dots,\frac12\right)\right).\] For $N=8$, this lattice is denoted by $E_8$. The inradius is given by $\rho_{D_N^+}= \frac{\sqrt{2}}{2}$. 
\end{remark} 
\begin{remark}[Leech lattice] \label{leech} The Leech lattice $\Lambda_{24}$ is a lattice in $\R^{24}$, see \cite[Chapter 4]{CS} for the precise construction. Its inradius is $1$.  It is the lattice which optimizes the sphere packing problem. 
\end{remark}
We recall the definition of fundamental domain of a lattice.
 \begin{definition}[Fundamental Domain]  
A fundamental domain  for the action of $G$ is   a set  which contains almost all  representatives for the orbits of $G$  and such  that the points whose orbit has more than one representative has measure zero, i.e. a measurable set $D\subseteq \R^N$ such that   $|D+g \cap  D|=0$ for every $g\in G$ with $g\neq id$, and $|\R^N\setminus (D+G)|=0 $.  \\
We will denote by $\mathcal D_G$ the set of all fundamental domains of $G$. 
\end{definition} 
If we fix a group $G$ and consider {\bf convex fundamental domains} associated to $G$, we obtain   the class of {\bf parallelohedra}, which are  centrally symmetric  polyhedra tiling the space by translations (see \cite{mcmullen}), with  at most  $2(2^N-1)$ facets.
A special parallelohedron associated to the  lattice $G$   is  its Voronoi cell, which is defined  as  
\[V_G:=\{x\in \R^N :\  |x|\leq |x-g| \qquad \forall g\in G, g\neq 0\}. \] 
  The distance between the center and every facet of the Voronoi cell associated to a lattice  is at least equal to $\rho_G$, so in particular the Voronoi cell contains a ball of radius $\rho_G$.  In a symmetric lattice  the facets of the Voronoi cell are all the same distance from the origin. 
 
Observe that not every parallelohedron is a Voronoi cell. In dimension $2$, two-dimensional parallelohedra (parallelograms and centrally symmetric hexagons) are
Voronoi cells  if and only if they are inscribed in a circle.  

\begin{remark} The permutahedron $P_N$  in $\R^{N+1}$ is a  $N$-dimensional polytope  given by the convex envelope of  the $(N+1)!$ points obtained by permutations of the coordinates of the point
\[
\left( -\frac N2,  -\frac N2+1,\ldots,\frac N2\right).
\] It  is contained in the hypersurface  $(1,\ldots,1)^\perp$, it is the Voronoi cell of the lattice $A_N^*$, and has exactly $2(2^N-1)$ facets. 

   \end{remark} 
   \begin{remark}[Hexagonal lattice in $\R^2$]\label{hex}
   The lattice $A_2$ is equivalent to the hexagonal lattice, whose Voronoi cell is the regular hexagon. The hexagonal lattice is generated by the vectors $(1,0)$ and $(-1/2, \sqrt{3}/2)$. In this form, the determinant is $\sqrt{3}/2$ and the  norm is $1$. So the inradius is $1/2$, whereas the covering radius is $1/\sqrt{3}$. 
   \end{remark}
 \begin{remark}[FCC and BCC lattices in $\R^3$]\label{fcc}
 The checkboard lattice $D_3$ is called  FCC (face- centered cubic) lattice, a family of generators is $(-1,-1,0),(1, -1, 0), (0,1,-1)$.  The volume is $2$, the norm is $\sqrt{2}$, and so the inradius is $1/\sqrt{2}$ and the covering radius is $1$. 
The Voronoi cell of the  FCC lattice is the rhombic dodecahedron. If centered at the origin it has $6$ vertices $(\pm 1, 0,0)$ and $8$ vertices $(\pm 1/2, \pm 1/2, \pm1/2)$.  

The dual of the root lattice $A_3^*$ is the BCC (body-centered cubic) lattice, a family of generators is $(2,0,0), (0,2,0), (1,1,1)$. The volume is $4$, the norm is $\sqrt{3}$ and so the inradius is $\sqrt{3}/2$ and the covering radius is $\sqrt{5}/2$.  The Voronoi cell of the BCC lattice is the truncated octahedron.  If centered at the origin it has $24$ vertices  $(\pm 1, \pm 1/2, 0)$.  

 \end{remark}

 We introduce a notion of convergence of lattices see \cites{cassels, ma}.  
 Note that, if $G$ is a lattice, then for every compact  subset $K\subset\R^N$, the set $G\cap K$ is finite.

 \begin{definition} 
 A sequence of lattices $G_k$ converges to $G$ if there exists for all $k$ a set of generators $g_k^i$ of $G_k$ such that $g_k^i\to g^i$, and $g^i$ is a set of generators of the lattice $G$.

A sequence of lattice $G_k$ converges in the Kuratowski sense to $G_i$, if
\[G=\{g\in \R^N \ | \limsup_{i\to +\infty} d(g, G_i)=0\}.\]
Note that $G$ is a closed subgroup of $(\R^N, +)$. 
Actually, the two notion of convergence are equivalent, see \cite[Section V.3, Theorem 1]{cassels}. 
 \end{definition} 
 
We recall the following compactness theorem for lattices due to Mahler \cite[Theorem 2]{ma} (see also \cite[Chapter V]{cassels}). 

\begin{theorem}\label{lemmaconvergence} 
Let $G_i$, $i\in \N$,  be a sequence of lattices and assume that there exist two constants $k,\delta>0$ such that  $\lambda(G_i)\geq \delta>0$ for all $i$ and $d(G_i)=|\R^N/G_i|\leq m$ for all $i$ and for some $m\in (0, +\infty)$.   Then there exists a subsequence $G_{i_n}$ and a lattice $G$ such that $G_{i_n}\to G$, and $\lambda(G)\geq \delta$, $d(G)\leq m$.
\end{theorem}
 
As a byproduct on this compactness theorem in \cite[Lemma 3.7]{cn}, see also \cite[Theorem 2.10]{cfn},  it has been proved the following result.
\begin{theorem}\label{lemmaconvergence2} 
Let $G_h$ be a sequence of lattices with $d(G_h)\leq m$ for every $h$ and  for some $m\in (0, +\infty)$.
Then, up to subsequences, $G_h\to G$ in the Kuratowski sense, where $G$ is either   a lattice with $d(G)\leq m$   or
a closed group which contains a line. In the second case, every sequence  $D_h$    of fundamental  domains  for $G_h$ converges to $\emptyset$  
  in $L^1_{loc}(\R ^N)$.  \end{theorem} 
  
 \section{Existence of minimizers} \label{seciso}
We fix a general class of perimeter functionals  $\Per(E)$ defined on measurable sets $E\subseteq \R^N$,  which are of two possible types.
\begin{itemize}\item Classical perimeter: 
for every measurable set $E\subseteq\R^N$, we define 
\begin{equation}\label{phiper} \Per(E)=\int_{\partial^* E}dH^{n-1}(x)\end{equation} where $\partial^* E$ is the reduced boundary of $E$ and $\nu(x)$ is the (measure theoretic) exterior normal at $E$ in $x$  (see \cite{maggi}). 
\item Fractional perimeter: let $s\in (0,1)$ and for every measurable set $E\subseteq\R^N$, we define 
\begin{equation}\label{sper}\Per(E)=\int_E\int_{\R^N\setminus E} \frac{1}{|x-y|^{N+s}}dxdy. \end{equation}
\end{itemize}

We fix a constant $m>0$  and we   introduce the following class of lattices 
\begin{equation}\label{gm} \mathcal{G}_{m}=\{G \text{ lattice in }\R^N, \text{ such that } \rho_G=m\}\end{equation}  
and moreover for a given lattice $G\in \mathcal{G}_m$ we consider the  class of fundamental domains which contain a ball of radius $m$: 
\begin{equation}\label{gm1} \mathcal{D}_G^{m}=\{D\in \mathcal{D}_G,  \text{ such that } \exists x\in D, |B_m(x)\setminus D|=0\}\end{equation}  
Note that this class is always not empty, since the Voronoi cell associated to a lattice $G\in \mathcal{G}_m$ is in $\mathcal D_G^m$. 

First of all we observe that, for every fixed group $G$,  every fundamental domain which minimizes the classical perimeter in \eqref{phiper} is indecomposable. 
\begin{lemma} \label{lemmaind} Let $G$ be a lattice and $E\in \mathcal{D}_G$ be a solution to the isoperimetric problem 
\[\inf_{D\in \mathcal D_G} \Per (D),\]
where $\Per$ is the classical perimeter. Then $E$ is  indecomposable.
\end{lemma} 
\begin{proof}  Let $E=\cup_i E^i$ where  $E^i$ are the indecomposable components of $E$ (see \cite{ac}),  and assume by contradiction that there exists $i$ such that $|E^i|\neq 0$ and $|F|\neq 0$ where  $F:=\cup_{j\neq i} E^j$. 

Then for every $g\in G$,  $\mathcal{H}^{N-1}(\partial^*E^i\cap \partial^*(F+g))=0$, where $\partial^* $ denotes the reduced boundary of a set (see \cite{maggi}): if it were not the case for some $g\in G$, then  $E^i\cup (F+g)$ would be a fundamental domain with $\Per(E^i\cup (F+g))<\Per(E)$.   Let us  consider the measurable sets $E^i+G= \cup_{g\in G} E^i+g$ and
 $F+G=\cup_{g\in G}F+g$. We have that $|(E^i+G)\cap (F+G)|=0$, $|\R^N\setminus ((E^i+G)\cup (F+G))|=0$, $|E^i+G|, |F+G|>0$  and $H^{N-1} (\partial^*(E^i+G)\cap\partial^*(F+G))=0$, which is not possible. 
\end{proof} 

We now prove  the following existence result. 

 \begin{theorem}\label{isothm} Let $\Per$ be a perimeter functional as in \eqref{phiper}, \eqref{sper}. 
  Then there exists a solution of the isoperimetric problem \begin{equation}\label{iso}
\inf_{G\in \mathcal{G}_m}\inf_{D\in \mathcal D_G^m} \Per (D).
\end{equation} 
 \end{theorem}

\begin{proof}The argument is divided in three steps.

\noindent {\bf Step 1}: First of all we consider a fixed $G\in\mathcal{G}_m$, as defined in \eqref{gm} and   we show that 
 there exists $E\in  \mathcal D_G$  such that
 \[\Per (E)=\inf_{D\in\mathcal{D}_G^m} \Per (D).\]

 Observe that since $G\in \mathcal{G}_m$, $d(G)=d_m>0$.  
 The existence of a fundamental domain of minimal perimeter for    a fixed lattice  has been obtained  in \cite[Theorem 3.2] {cn},  with  a     concentration compactness  argument dating back to Almgren and by lower semicontinuity of the perimeter functional (see \cite{alm} and the monograph \cite{maggi}).   
 
  In particular given a minimizing sequence of fundamental domains $D_k\in \mathcal{D}_G^m$, then     there exist $g_k^i\in G$, for  $i\in I\subseteq \N$,  and $E^i\subseteq \R^N$  such that 
 $|g_k^i-  g_k^j|\to +\infty$ if $j\neq i$ as $k\to +\infty$,  and  $ D_k+g_k^{i}\to E^i$ locally in $L^1$ as $k\to +\infty$. Moreover  $\cup_i E_i\in \mathcal{D}_G$: so  $|\cup_i E^i|= |D_k|=d_m$  and $\Per(\cup_i E_i)=\sum_i\Per(E_i)=\inf_{D\in \mathcal{D}_G^m} \Per(D)$.
 
 \vspace{0.3cm}
 
\noindent {\bf Step 2}:  We  prove that  $E \in \mathcal{D}_G^m$ (so it contains a ball of radius $m$). 

If the perimeter functional is the fractional perimeter \eqref{sper}  then there exists a unique $i$ such that $|E^i|\neq 0$. This is an immediate consequence of the fact that $\Per (\cup_i E^i)<\sum_i \Per(E^i)$ if there are at least two elements 
$|E^i|, |E^j|\neq 0$, and on the other hand by Step 1 we have that $\Per(\cup_i E^i)=\sum_i\Per(E^i)$ (see \cite[Remark 3.6]{cn}).   Since $D_k\to E^i$ locally in $L^1$ and $|D_k|=|E^i|$, we conclude that actually the convergence is in $L^1(\R^N)$, which implies that  $E^i\in \mathcal{D}_G^m$.

If the perimeter functional is the local perimeter \eqref{phiper}, by Lemma \ref{lemmaind} $E$ is indecomposable. So  $D_k\to E$ in $L^1(\R^N)$, and  we conclude as above.  
  
 \vspace{0.3cm} 
 
\noindent {\bf Step 3}: we minimize among all possible lattices $G\in \mathcal{G}_m$.

 Let us fix a minimizing sequence, that is a sequence of lattices $G_n\in \mathcal{G}_m$ and a sequence of fundamental domains $D_n\in \mathcal{D}_{G_n}^m$.  So $\Per(D_n)\leq \inf_{G\in \mathcal{G}_m}\inf_{D\in \mathcal D_G^m} \Per (D)+1:=C$, and by the isoperimetric inequality this implies $|D_n|\leq K$, and so $d(G_n)\leq K$ for every $n$. Moreover by the constraint $|D_n|\geq  |B_m|= m^N\omega_N$. 

By Theorem \ref{lemmaconvergence2}, ut to subsequences either $G_n\to G$ where $d(G)\leq K$, or $D_n\to \emptyset$ in $L^1_{loc}$. 
 
 By the same argument of concentration compactness recalled in Step 1, see \cite[Lemma 3.4]{cn}, 
since $D_n$ are measurable sets with  $|D_n|\geq  m^N\omega_N$ and $\Per(D_n)\leq C$, up to passing to a subsequence,  there exist  $z_h^i\in \Z^N$ ($i\in\N, h\in\N$),  with $|z^i_h-z_h^j|\to +\infty$ as $h\to +\infty$ for $i\neq j$, and a family $(D^i)$ of measurable sets in $\R^N$ such that 
$D_h-z_h^i\to D^i$  in $L^1_{\rm loc}(\R^N)$ and $\sum_i |D^i| =m$.  This implies that it is not possible that $D_n\to \emptyset$ locally in $L^1$, and so, up to a subsequence,  $G_n\to G$ where $G$ is a lattice. 

Then we proceed with a similar proof as in Step 1, see the details in \cite[Lemma 3.1, Theorem 1.2]{cfn}.
  \end{proof} 
  
  For a fundamental domain $D$ associated to a lattice we can introduce the ratio 
\[\mathcal{I}(D)= \frac{\Per D}{\rho^2_D}\qquad \text{ where }\rho_D=\max\{r\geq 0: \ \exists x\in D, \  |D\setminus B(x, r)|=0 \}.\] 
Note that problem \eqref{iso} is equivalent to the following minimum problem: 
\begin{equation}\label{isonuovo} 
\min_{G \text{ lattice in }\R^3}\min_{D\in \mathcal{D}_G}\mathcal{I}(D).  
\end{equation}

 \begin{remark} The same result as Theorem \ref{isothm} holds under a more general class of constraints. 
 
Let us fix a nonnegative functional $\mathcal{F}$ defined on measurable sets, a constant $m>0$  and   the following class 
\[\mathcal{G}_{m}=\{G \text{ lattice in }\R^N, \text{such that }\exists D\in \mathcal{D}_G, \mathcal{F}(D)=m\}.\] The  isoperimetric problem   reads:
\begin{equation}\label{iso1}
\inf_{G\in \mathcal{G}_m}\inf_{D\in \mathcal D_G, \mathcal{F}(D)=m} \Per (D).
\end{equation} 
The same argument as in the proof of Theorem \ref{isothm} gives the existence of a solution to \eqref{iso1} f we assume that  $\mathcal{F}$ satisfies the following conditions: 
 \begin{itemize}
\item $\mathcal{F}$ is continuous in $L^1$, that is if $E_i\to E$ in measure (that is $\chi_{E_i}\to \chi_E$ in $L^1$), then $\mathcal{F}(E_i)\to \mathcal{F}(E)$;
\item  $\mathcal{F}$ is nondegenerate in the sense that   for any $C>0$ there exists a constant $k(C)=k>0$ such that if  $\mathcal{F}(E)\geq C$ for $E\in\mathcal{M}$, then $|E|\geq k$.   \end{itemize}

The case of the inradius is obtained by considering \[\mathcal{F}(E)= \sup\{r>0, \text{ such that }|B_r(x)\setminus E|=0\text{ for some $x\in E\}$}.\] 
Another possible functional satisfying the previous assumptions is the Riesz energy (see \cite{lieb}):  \[\mathcal{F}(E):=\int_{E}\int_{E}\frac{1}{|x-y|^{N-\alpha}} dxdy,\qquad \text{for some $\alpha\in (0, N)$. }\] 
\end{remark}

 \section{Regularity} \label{secreg}
 In order to study the regularity of the minimal fundamental domain we introduce the notion of partitions.
 \begin{definition}
A partition of $\R^N$ is a collection of measurable subsets $\{E_k\}_{k\in\mathbb I}$,  where $\mathbb I$ is either a finite or a countable set of ordered indices, such that 
\begin{enumerate} 
\item $|E_k|>0$ for all $k$,
\item $|E_k\cap E_j|=0$ for all $k\ne j$,
\item $|\R^N\setminus\cup_k E_k|=0$.
\end{enumerate}
\end{definition}
A fundamental domain $D$  of a lattice $G$ naturally induces the $G$-periodic partition $\{E_g\}_{g\in G}$, where $E_g = g D$. Moreover if $D$ is bounded, then the associated partition is locally finite. 

%
%
%
%
\begin{definition}\label{defminimi}
We say that a partition $\{E_k\}_k$ is $\Lambda$-minimal in an open set $A\subset \R^N$, for some $\Lambda\ge 0$, if
$\sum_k \Per(E_k,A)<+\infty$ and 
$$
\sum_k \Per(E_k,A)\le \sum_k \big[\Per(F_k,A) + \Lambda \mu(E_k\Delta F_k)\big],
$$ 
for every partition $\{F_k\}_k$ of $X$ such that $E_k\Delta F_k \Subset A$ for all $k$.\\
We say that the partition is $(\Lambda,r)$-minimal (see \cite{maggi}) for some $\Lambda\ge 0$ and $r>0$, 
if  it is $\Lambda $-minimal in $B_r(x)$ for all $x\in \R^N$.\\
\end{definition}


 \begin{theorem} \label{regthm} 
 Let $D$ be a solution to \eqref{iso} with $\rho_G=1$  and $B\subset D $ is the ball of radius $1$. 
Then the partition $D+G$ is  $(\Lambda,r)$-minimal   in $\R^N \setminus (B+G)$ 
 for every $r\leq r_0<1=\rho_G$,  
with $\Lambda = 0$ in the case of the local perimeter, and $\Lambda=\frac{\omega_N}{s (2-2r_0)^s}=\int_{|h|>  2-2r_0} |h|^{-s-N} dh$,  in the case of the nonlocal perimeter. 

Then, there exists  a closed singular set $\Sigma\neq \emptyset$ 
with $\mathcal{H}^{N-1}(\Sigma)=0$,  such that $\partial D\setminus \Sigma$ is a $C^{1,\frac{1+s}{2}}$ hypersurface, where   $s\in (0,1)$ is the order of the fractional perimeter in the case \eqref{sper} and $s=1$ in the case of local perimeter \eqref{phiper}. 
  \end{theorem} 
 \begin{proof} The first part of the statement follows the same argument of the proof of   $(\Lambda,r)$-minimality   in $\R^N $ of  partitions associated to fundamental domains with minimal perimeter (without assuming the constraint  on the inradius), given in \cite[Proposition 4.1]{cn}. 
 
 As for the regularity of the boundary of $D$, we may refer to classical regularity results about (local) minimal surfaces with obstacles \cite{kinderlehrer, caff}, whereas for nonlocal minimal surfaces with obstacle we refer to \cite[Theorem 1.2]{caff2}. 
 \end{proof} 

\section{Comparison with other tilings} 
  
\subsection{The rhombic  dodecahedral honeycomb and the Kelvin foam in $\R^3$} \label{kfoam}
In this section we consider the isoperimetric problem \eqref{iso} for the classical perimeter in dimension $N=3$. 
A  restricted  version of this isoperimetric problem has been already discussed  in the literature, by considering as admissible competitors only convex fundamental domains containing a ball of given radius.
In particular  in \cite{bez3, bez2} (see also \cite{langi}) it has been posed  the problem of finding  the parallelohedron  with minimal perimeter containing a ball of radius $1$, stating  the following conjecture:\\ 
  {\bf Rhombic Dodecahedral Conjecture}: The perimeter of any parallelohedron containing a ball of  radius $1$  is at least that of a regular rhombic dodecahedron. 

%
  
Going back to our problem, the first observation is that   the  regular rhombic  dodecahedron  cannot be a solution. 
\begin{proposition}\label{nomin} 
  The  regular rhombic  dodecahedron  is not a solution to \eqref{iso} in dimension $N=3$ with the standard perimeter.
  
\end{proposition}  
  \begin{proof} 
Let $D$ be a solution of the isoperimetric problem \eqref{iso} for the standard perimeter  and inradius $1$ in dimension $N=3$. 
Let us fix $x\in \R^N$ and $r<1=\rho_G$ and consider the partition $D+G$ inside a ball $B_r(x)$. By Theorem \ref{regthm},  this partition inside the ball is locally minimal for the perimeter and so either it is   empty or  a minimal   cone in $\R^3$. The  complete classification of minimal singular cones is available in dimensions $3$ (see \cite{taylor}):   three half-hyperplanes meeting at 120 degrees or the cone over the $1$-skeleton of a tetrahedron. So every minimal cone either is an hyperplane or has to coincide with one of the previous singular cones: these  necessary conditions for minimality are called Plateau's conditions. 

Regular rhombic  dodecahedron has two types of vertices which are  threefold and fourfold.
If we look at the partition generated by a regular rhombic dodecahedron  inside a ball  $B_r(x)$ with $r<1$ and $x\in\R^N$ one of the fourfold vertices, we obtain a cone with  more than $4$ chambers, so it cannot coincide with any of the previous configurations.  On the other hand if $x$ is one of the threefold vertices, we obtain  the cone over the $1$-skeleton of a cube, which again is not minimal. 
 So regular rhombic  dodecahedron cannot be minimal. 
\end{proof} 

\begin{remark} \upshape The same argument as in Proposition \ref{nomin} shows that  the solution to \eqref{iso} cannot be a parallelohedron, since parallelohedra never satisfy   Plateau's conditions at vertices and edges. \end{remark}

Besides the regular rhombic dodecahedral honeycomb there is  another important structure arising in the Kelvin problem, which is  the variational problem corresponding to the surface minimizing partition into cells of equal volume \cite{thompson}. 

The construction of Lord Kelvin in \cite{thompson} produces a partition of the space which is  a critical point for the Kelvin problem.  
As describes in detail in \cite[Section 7]{w},
Lord Kelvin first considered the partition generated by the regular rhombic dodecahedron, which as shown in Proposition \ref{nomin} cannot be a local minimizer.  The idea was then to substitute the partition locally around any  fourfold vertex with the Plateau minimal partition with boundary given by the $1$-skeleton of a cube. This introduces a new flat facet with four curved edges,  which may be directed along any of the three cubic axes (and one may choose the same cubic axis for all of them).  
So he was attaching many copies of Plateau’s wire cube together, edge to edge and  in  the rhombohedral cell  a  square
face is introduced at the top and bottom and new horizontal edges appear at the
four fourfold vertices that lie on the equator. These are the edges of square faces belonging to adjacent cells.
The resulting  cell, that we may call {\bf anisotropic Kelvin cell},  is a  candidate structure which satisfies Plateau's conditions
and is a fundamental domain for the FCC lattice, like the  regular rhombic dodecahedron.

Since this cell is not isotropic, 
Lord Kelvin further modified it by changing the ratio between the axis in the  Plateau’s wire cube, taking as a boundary to construct the minimal surface a parallelepiped and not a cube. 
By varying this ratio from $1$ to $1/\sqrt{2}$ he obtained an isotropic cell, that we call  {\bf   Kelvin cell} $\mathcal{T}_K$. This cell actually is space filling when centered at the points of the BCC lattice, even if it is not the Voronoi cell of such lattice (which is the regular truncated octahedron): in particular it is a fundamental domain of the BCC lattice and it contains a ball of radius less or equal to the ball contained in the regular truncated octahedron of the same volume. Indeed  it can be also obtained as a small  deformation of the faces of   the regular truncated octahedron, see \cite{prince} for a description of this construction, and it  has six planar quadrilateral
faces and eight nonplanar hexagons of zero curvature. 
 If we denote with $\mathcal T$ the truncated octahedron we get that
 \[\rho_{\mathcal T_K}\leq \rho_{\mathcal T}\] and so
 \begin{equation}\label{octo}\mathcal{I}(\mathcal{T}_K)\geq \mathcal{I}(\mathcal T) \frac{\Per(\mathcal{T}_K)}{\Per(\mathcal T)}\sim 0.998 \   \mathcal{I}(\mathcal T),\end{equation}
 by using the computation of the ratio between the perimeter of the Kelvin cell and of the regular truncated octahedron obtained in \cite{prince}. 
 

\begin{proposition}  The Kelvin cell $\mathcal{T}_K$ is not a solution to \eqref{iso}. 
\end{proposition} 

\begin{proof} In order to show the result, it is sufficient to show that $\mathcal{I}(\mathcal{T}_K)>\mathcal{I}(\mathcal D)$, where   $\mathcal D$ is the regular rhombic dodecahedron.  

We have that the  perimeter of the regular rhombic dodecahedron  $\mathcal D$ of inradius $1$ is given by $12\sqrt{2}$, 
see \cite{bez, hales}, so in particular
\[\mathcal{I}(\mathcal D)= 12\sqrt{2}\sim 16.97.\] 
On the other hand the inradius  of the regular truncated octahedron $\mathcal T$ of volume $1$   is given by  $\rho=\frac{ \sqrt{3}}{2^{\frac53}}$ (see \cite[Chapter 2]{CS} and its perimeter $ \frac{6\sqrt{3}+3}{2^{\frac43}}$ (see \cite{prince}).
So 
\[\mathcal{I}(\mathcal{T})= \frac{ 6\sqrt{3}+3}{2^{\frac43}} \frac{2^{\frac{10}{3}}}{3}= 4(2\sqrt{3}+1)=17.856, 
\]  
hence by \eqref{octo} we conclude 
\[
\mathcal{I}(\mathcal{T}_K)\ge 0,998 \cdot 17.856 \sim 17.82 > \mathcal{I}(\mathcal{D}).
\]
\end{proof} 

It remains open the problem of identify a solution to problem \eqref{iso} in $\R^3$. 
Since the Kelvin cell is not a solution, and in view of the Strong Dodecahedral Theorem proved by Hales,
one might be led to state the following:\\
{\bf FCC Conjecture}:   the optimal lattice for problem \eqref{iso} is FCC.

  \subsection{The $24$–cell honeycomb in $\R^4$}\label{24cell}
%
%
In dimension $N=4$,  there exists a well known regular tessellation, which is  the $24$-cell honeycomb.
The $24$-cell  is the convex hull of its vertices which can be described as the $24$ coordinate permutations of $(\pm 1,\pm1, 0,0)$.  
 In this frame of reference the $24$-cell has  inradius    $ \frac{\sqrt{2}}{2}$.   
 Up to a rescaling, the $24$-cell is  the Voronoi cell of the $D_4$ lattice.
 
If a sphere is inscribed in each $24$ cell of the $24$-cell honeycomb, the resulting arrangement is the densest known regular sphere packing in four dimensions, with kissing number $24$, even if the sphere packing problem  is still unsolved in dimension $4$.  
This suggests that the $24$-cell could be a solution to problem \eqref{iso}. 

In this direction, we observe that 
 the $24$-cell honeycomb satisfies the following local minimality property, which is a necessary condition for being a solution to problem \eqref{iso}, as stated in Theorem \ref{regthm}. 

\begin{proposition} 
The  $24$-cell honeycomb  is a minimizer for the perimeter in every ball $B_r(x)$, with $x\in \R^4$ and $0<r<\frac{\sqrt{2}}{2}$.  
In particular, it is a critical point of problem \eqref{iso}.
\end{proposition} 

\begin{proof} 
The thesis follows from the fact that inside a ball $B_r(x)$ with $r$ smaller than the inradius, the $24$-cell honeycomb is either empty or coincides, up to a translation, with 
one of the following cones in $\R^4$: a hyperplane, three half-hyperplanes meeting at 120 degrees, 
the cone over the $1$-skeleton of a tetrahedron times $\R$, the cone over the $2$-skeleton of a hypercube.
Thanks to the results by Taylor \cite{taylor} and Brakke  \cite{brakke} it is known that such cones define partitions of $\R^4$ which are minimal for the perimeter under compact perturbations.
\end{proof}  

Due to the properties of the $24$-cell,  in analogy with the rhombic dodecahedron conjecture in $\R^3$,
we may state  the following conjecture: \\
{\bf $24$–cell Conjecture}:  The perimeter of any fundamental domain of a lattice in $\R^4$, containing a ball of  radius $1$,  is at least that of a 24–cell of inradius $1$. 


 \end{document}